\documentclass[12pt]{article}
\usepackage{amssymb}
\usepackage{amsmath,amsthm}
\usepackage{verbatim}
\usepackage{graphicx}
\usepackage{epstopdf}
\usepackage{epsf}
\usepackage{geometry}
\usepackage[flushmargin]{footmisc}
\usepackage{subfig}
\usepackage[usenames,dvipsnames]{color}

\usepackage[colorlinks=true,citecolor=black,linkcolor=black,urlcolor=blue]{hyperref}

\usepackage{enumitem}
\setlist{itemsep=0pt, topsep=0pt}

\newtheorem{theorem}{Theorem}[section]
\newtheorem{lemma}[theorem]{Lemma}

\newtheorem{example}[theorem]{Example}

\newtheorem{conjecture}[theorem]{Conjecture}
\newtheorem{corollary}[theorem]{Corollary}

\newtheorem{proposition}[theorem]{Proposition}
\newtheorem{observation}[theorem]{Observation}
\newtheorem{problem}[theorem]{Problem}

\numberwithin{subcase}{case}

\newtheorem{definition}[theorem]{Definition}

\newcommand{\ep}{\epsilon} 
\newcommand{\eps}{\epsilon} 

\newcommand{\floor}[1]{\left\lfloor#1\right\rfloor}
\newcommand{\ceiling}[1]{\left\lceil#1\right\rceil}

\newcommand{\of}[1]{\left( #1 \right)}
\newcommand{\set}[1]{\left\{ #1 \right\}}

\newcommand{\sqbs}[1]{\left[ #1 \right]}

\newcommand{\bfrac}[2]{\of{\frac{#1}{#2}}}
\renewcommand{\Pr}[1]{\mathbb{P}\sqbs{#1}}
\newcommand{\E}[1]{\mathbb{E}\sqbs{#1}}
\newcommand{\tbf}[1]{\textbf{#1}}
\newcommand{\wh}[1]{\widehat{#1}}
\newcommand{\opoo}{(1+o(1))}

\newcommand{\gnp}{G({n,p})}

\newcommand{\cefrac}[2]{\ceiling{\frac{#1}{#2}}}
\newcommand{\flfrac}[2]{\floor{\frac{#1}{#2}}}

\newcommand{\cT}{{\mathcal T}}
\newcommand{\cP}{{\mathcal P}}
\newcommand{\cC}{{\mathcal C}}

\newcommand{\tp}{\mathrm{tp}}
\newcommand{\tc}{\mathrm{tc}}
\newcommand{\pp}{\mathrm{pp}}

\newcommand{\cp}{\mathrm{cp}}

\newcommand{\tm}{\mathrm{tm}}

\title{Partitioning random graphs into monochromatic components}

\author{Deepak Bal\\
\small Department of Mathematical Sciences\\[-0.8ex]
\small Montclair State University\\[-0.8ex] 
\small Montclair, New Jersey; U.S.A.\\
\small\tt deepak.bal@montclair.edu\\
\and
Louis DeBiasio\thanks{Research supported in part by Simons Foundation Collaboration Grant \# 283194.}\\
\small Department of Mathematics\\[-0.8ex]
\small Miami University\\[-0.8ex] 
\small Oxford, Ohio; U.S.A.\\
\small\tt debiasld@miamioh.edu\\
}

\begin{document}

\maketitle

\begin{abstract}
Erd\H{o}s, Gy\'arf\'as, and Pyber 
(1991) 
conjectured that every $r$-colored complete graph can be partitioned into at most $r-1$ monochromatic components; this is a strengthening of a conjecture of Lov\'asz 
(1975) 
and Ryser 
(1970) 
in which the components are only required to form a cover.  An important partial result of Haxell and Kohayakawa (1995) 
shows that a partition into $r$ monochromatic components is possible for sufficiently large $r$-colored complete graphs.

We start by extending Haxell and Kohayakawa's result to graphs with large minimum degree, then we provide some partial analogs of their result for random graphs.  In particular, we show that if $p\ge \bfrac{27\log n}{n}^{1/3}$, then a.a.s.~in every $2$-coloring of $G(n,p)$ there exists a partition into two monochromatic components, and for $r\geq 2$ if $p\ll \bfrac{r\log n}{n}^{1/r}$, then a.a.s.~there exists an $r$-coloring of $G(n,p)$ such that there does not exist a cover with a bounded number of components. Finally, we consider a random graph version of a classic result of Gy\'arf\'as 
(1977) 
about large monochromatic components in $r$-colored complete graphs.  We show that if $p=\frac{\omega(1)}{n}$, then a.a.s.~in every $r$-coloring of $G(n,p)$ there exists a monochromatic component of order at least $(1-o(1))\frac{n}{r-1}$.

\end{abstract}

\section{Introduction}
For a graph $G$ and positive integer $r$, the \tbf{$r$-color tree-partition (tree-cover) number} of $G$, denoted by $\tp_r(G)$ ($\tc_r(G)$), is the minimum $s$ such that for every $r$-edge-coloring of $G$, there exists a collection of monochromatic connected subgraphs $\{H_1, \dots, H_t\}$ with $t\leq s$ such that $\{V(H_1),\dots, V(H_t)\}$ forms a partition (cover) of $V(G)$; as each subgraph $H_i$ contains a monochromatic spanning tree we use ``connected subgraph'', ``tree'', and ``component'' interchangeably throughout the paper.  Similarly define $\pp_r(G), \cp_r(G)$ to be the $r$-color path-partition number and $r$-color cycle-partition number of $G$ respectively. 

Gy\'arf\'as \cite{Gy77} noted that the following is an equivalent formulation of what is known in the literature as ``Ryser's conjecture" or the ``Lov\'asz-Ryser conjecture."

\begin{conjecture}[Ryser 1970 (see \cite{Hen}), Lov\'asz 1975 \cite{Lov}]\label{con:LovRys}
Let $r\geq 2$.  For all graphs $G$, $\tc_r(G)\leq (r-1)\alpha(G)$.
\end{conjecture}

If true, this conjecture is best possible when $r-1$ is a prime power by a well known example using affine planes\footnote{In an affine plane of order $r-1$, there are $r$ parallel classes of $r-1$ lines each. To each of the $r$ parallel class assign a distinct color.}. For $r=2$, this is equivalent to the K\H{o}nig-Egerv\'ary theorem.  Aharoni \cite{Ah} proved the $r=3$ case, and for $r\geq 4$ it is open.  Slightly more is known in the case $\alpha=1$ (i.e.~when $G=K_n$), where it has been proved for $r\leq 5$ (see \cite{GySurvey} and \cite{FLM} for more details).

In a seminal paper, Erd\H{o}s, Gy\'arf\'as, and Pyber \cite{EGP} proved that for all $r\geq 2$, $$\tp_r(K_n)\leq \pp_r(K_n)\leq \cp_r(K_n)=O(r^2\log r)$$ and made the following conjecture.

\begin{conjecture}[Erd\H{o}s, Gy\'arf\'as, Pyber 1991]\label{con:EGP}
For all $r\geq 2$, $\tp_r(K_n)=r-1$, $\pp_r(K_n)=\cp_r(K_n)=r$.
\end{conjecture}

For countably infinite complete graphs, Conjecture \ref{con:EGP} is known to be true for paths and cycles for all $r$, and known to be true for trees when $r=2,3$ (with the appropriate notion of paths, cycles, and trees).  Rado \cite{Rad} proved $\pp_r(K_{\mathbb{N}})=r$; Elekes, D. Soukup, L. Soukup, and Szentmikl\'ossy \cite{ESSS} proved $\cp_r(K_{\mathbb{N}})=r$; and Nagy and Szentmikl\'ossy (see \cite{EGP}) proved $\tp_3(K_{\mathbb{N}})=2$.  However, for finite complete graphs the story is more complicated.

For trees, an old remark of Erd\H{o}s and Rado says that a graph or its complement is connected, i.e.~$\tp_2(K_n)=1$.  Erd\H{o}s, Gy\'arf\'as, and Pyber \cite{EGP} proved that $\tp_3(K_n)=2$.  Later, Haxell and Kohayakawa \cite{HK} proved the following. 

\begin{theorem}[Haxell, Kohayakawa 1995]\label{thm:HK}
Let $r\geq 2$.  If $n\geq \frac{3r^4r!\log r}{(1-1/r)^{3(r-1)}}$, then $\tp_r(K_n)\leq r$. 
\end{theorem}

Later, Fujita, Furuya, Gy\'arf\'as, and T\'oth \cite{FFGT} conjectured that the partition version of Conjecture \ref{con:LovRys} is true and proved it in the case when $r=2$.


For paths, Gerencs\'er and Gy\'arf\'as gave a simple proof of $\pp_2(K_n)=2$ (see the footnote in \cite{GG}).  
Much later, Pokrovskiy \cite{P} proved $\pp_3(K_n)=3$.

For cycles, Lehel conjectured that $\cp_2(K_n)=2$ (in fact, with cycles of different colors).  This was proved for large $n$ by \L uczak, R\"odl, and Szemer\'edi \cite{LRS}, and then for smaller, but still large $n$ by Allen \cite{Al}, and finally for all $n$ by Bessy and Thomass\'e \cite{BT}.  For general $r$, Gy\'arf\'as, Ruszink\'o, S\'ark\"ozy, Szemer\'edi \cite{GyRSS1} improved the result from \cite{EGP} to 
$$\pp_r(K_n)\leq \cp_r(K_n)\leq 100r\log r.$$ 
However for $r\geq 3$, Pokrovskiy \cite{P} proved $\cp_r(K_n)>r$.

\subsection{Large minimum degree}

Motivated by a new class of Ramsey-Tur\'an type problems raised by Schelp \cite{Schelp}, Balogh, Bar\'{a}t, Gerbner, Gy\'arf\'as, and S\'ark\"ozy \cite{BBGGS} conjectured (and proved an approximate version of) a significant strengthening of Bessy and Thomasse's result.  That is, if $\delta(G)>3n/4$, then $\cp_2(G)\leq 2$ (with cycles of different colors); they also provided an example which shows that the conjecture would be best possible.  DeBiasio and Nelsen \cite{DN} proved that this holds for $G$ with $\delta(G)>(3/4+o(1))n$ and then Letzter \cite{L2} proved that it holds exactly for sufficiently large $n$.

In Observation \ref{obs:alphar}, we note that there are graphs with minimum degree $n-r$ for which $\tp_r(G)\geq \tc_r(G)\geq r$ and thus it is natural to wonder how small we can make $\delta(G)$ while maintaining $\tp_r(G)\leq r$.  In Theorem \ref{thm:mindegpartition}, we prove a strengthening of Theorem \ref{thm:HK} for graphs with large minimum degree.  A corollary of our result is the following.

\begin{corollary}
For all $r\geq 2$ there exists $n_0$ such that if $G$ is a graph on $n\geq n_0$ vertices with $\delta(G)> (1-\frac{1}{er!})n$, then $\tp_r(G)\leq r$.
\end{corollary}

Furthermore, in Example \ref{obs:tcrlb} we show that the minimum degree in the above result cannot be improved beyond $\sim (1-\frac{1}{r+1})n$.  Additionally, in Theorem \ref{thm:2colorcover} we show that for \emph{covering} 2-colored graphs with two monochromatic trees, this lower bound on the minimum degree is tight.

Theorem \ref{thm:mindegpartition} actually gives a ``robust" tree partition; that is, a collection of trees together with a linear sized set $L$ such that after deleting any subset of $L$, the remaining graph has a tree partition.  This is important as we will use it to obtain results on the tree partition number of the random graph $\gnp$.

Finally, as a consequence of our method of proof, we are able to improve the bound on $n$ in Theorem \ref{thm:HK}.

\begin{theorem}\label{thm:completepartition}
For $r\geq 2$ and $n\geq 3r^2r!\log r$, $\tp_r(K_n)\leq r$. 
\end{theorem}

\subsection{Random graphs}

An active area of current research concerns sparse random analogs of combinatorial theorems (see the survey of Conlon \cite{Con}).  An early example of such a result is the so called Random Ramsey Theorem.
Say $G\to_r H$ if every $r$-coloring of $G$ contains a monochromatic copy of $H$.  For fixed graphs $H$, R\"odl and Ruci\'nski \cite{RR1} determined the threshold for which a.a.s.\footnote{We say that a sequence of events $A_n$ happens \tbf{a.a.s.} if $\lim_{n\to\infty}\Pr{A_n} \to 1$.}~$\gnp \to_r H$.  For the case of paths, Letzter \cite{L1} proved that for $p=\frac{\omega(1)}{n}$ a.a.s., $\gnp \to_2 P_{(2/3-o(1))n}$. 
Random analogs of asymmetric Ramsey problems, hypergraph Ramsey problems, and van der Waerden's Theorem have also been studied (again, see \cite{Con}).

In light of these results, it is a natural question to ask whether monochromatic partitioning problems can be extended to the realm of random graphs in an interesting way. Towards this, we prove the following results which provide partial analogs of Conjecture \ref{con:LovRys}, 
and Theorem \ref{thm:HK} for random graphs.

\begin{theorem}\label{thm:gnp_threshold}
For all $r\geq 2$, there exists $C\geq r$ such that a.a.s.
\begin{enumerate}
\item  ~if~  $p\ge \bfrac{27\log n}{n}^{1/3}$ ~then~ $\tp_2(G(n,p)) \le 2$, and
\item  ~if~ $p\geq \bfrac{C\log n}{n}^{1/{(r+1)}}$, ~then~ $\tc_r(\gnp )\leq r^2$, and
\item \label{thm:gnp_old_one} ~if~ $p\ge \bfrac{C\log n}{n}^{1/r}$, then there is a collection of $r$ vertex disjoint monochromatic trees which cover all but at most $9r\log n /p=O(n^{1/r}(\log n)^{1-1/r})$ vertices.
\end{enumerate}
\end{theorem}

\begin{theorem}\label{thm:gnp_threshold_low}
For all $r\geq 2$,
\begin{enumerate}
\item  ~if~ $p=\bfrac{r\log n-\omega(1)}{n}^{1/r}$, ~then~ $\tc_r(\gnp )> r$, and
\item ~if~ $p = o\of{ \bfrac{r\log n}{n}^{1/r}}$, ~then~ $\tc_r(\gnp )\to \infty$. 
\end{enumerate}
\end{theorem}

It is interesting to compare Conjecture \ref{con:LovRys} to Theorem \ref{thm:gnp_threshold_low} as our results imply that almost every graph $G\sim G(n,1/2)$ (the uniform distribution on all graphs with $n$ vertices) satisfies $\tc_r(G)\leq r^2$, which is much smaller than the conjectured upper bound of $(r-1)\alpha(G)$ since it is known (see e.g.~\cite{FK}) that a.a.s.~$\alpha(G) \sim 2\log_2 n$.  So not only are tightness examples rare, examples for which $\tp_r(G)\geq \tc_r(G)>r^2$ are rare.

\subsection{Large monochromatic components}

We consider one further related line of research. Note that if an $r$-colored graph $G$ can be covered by $t$ monochromatic components, then $G$ contains a monochromatic component of order at least $|V(G)|/t$.  So we may directly ask how large of a monochromatic component we may find in an $r$-colored graph\footnote{Historically, the monochromatic partitioning problems mentioned in the first part of the introduction were motivated by their implications for graph Ramsey problems (see \cite{P2}).}.  Given a positive integer $r$ and a graph $G$, we let $tm_r(G)$ be the maximum integer $s$ such that the following holds: in every $r$-coloring of the edges of $G$, there exists a monochromatic component with at least $s$ vertices.  For $r\geq 2$, Gyarf\'as \cite{Gy77} proved $\tm_r(K_n)\geq \frac{n}{r-1}$ and F\"uredi \cite{Fur} proved $\tm_r(G)\geq \frac{n}{(r-1)\alpha(G)}$ for all graphs $G$ (see Theorem 5.6 in \cite{GySurvey}).  Furthermore, this is tight when $r-1$ is a prime power using the same affine plane example mentioned before.  Given the discussion above, note that F\"uredi's result would be implied by Conjecture \ref{con:LovRys}. 

Concerning random graphs, two sets of authors \cite{Spo}, \cite{Boh} independently found the threshold for $\tm_r(\gnp) = \Theta(n).$ Specifically, they prove that there exists an analytically computable constant $\psi_r$ such that if $c<\psi_r$, then $\tm_r(G(n,c/n)) = o(n)$ and if $c > \psi_r$ then $\tm_r(G(n,c/n)) = \Omega(n)$.

We prove\footnote{Essentially the same result was independently discovered by Dudek and Pra{\l}at \cite{DuPr}.} the following random analog of the fact that $\tm_r(K_n)\geq \frac{n}{r-1}$. 

\begin{theorem}\label{thm:randomlargemono}
For all $r\geq 2$ and sufficiently small $\ep>0$, there exists $C$ such that for $p\geq \frac{C}{n}$, a.a.s.~every $r$-coloring of $\gnp$ contains either a monochromatic tree of order at least $(1-\ep)n$ or a monochromatic tree with at least $(1-\ep)\frac{n}{r-1}$ leaves, which implies $\tm_r(\gnp )\geq (1-\ep)\frac{n}{r-1}$.
\end{theorem}

Again, it is interesting to compare Theorem \ref{thm:randomlargemono} to the corresponding deterministic version, as our result implies that almost every graph $G$ satisfies $\tm_r(G)\geq (1-\ep)\frac{n}{r-1}$, which is much larger than the bound of $\frac{n}{(r-1)\alpha(G)}$ given by F\"uredi's result for which there are examples showing tightness.

\section{Overview and notation}

\subsection{Overview}
We consider large minimum degree versions and random versions of some classic results for edge colored complete graphs.  In certain cases we will use the large minimum degree results together with the sparse regularity lemma to obtain results for random graphs.  In these cases our approach is as follows:  First, prove that edge colored graphs of high minimum degree contain (a robust version of) the desired structure.  Second, applying the sparse regularity lemma to the random graph gives a reduced graph with high minimum degree and thus we can apply the high minimum degree result.  This structure in the reduced graph corresponds to an approximate spanning structure in the original graph. As a simple application of this approach we obtain Theorem \ref{thm:randomlargemono}.  

A less standard application of this approach is given in the proof of Theorem \ref{thm:gnp_threshold}(iii) where we are trying to improve the exponent from $1/(r+1)$ to $1/r$.  We use the method of multiple exposures to build a tree cover while maintaining a set of vertices which are leaves in each of the monochromatic trees.  On each step, the leaf set shrinks by a factor of $p$ and at the end of the possibly $r$ steps, we require the leaf set to contain more than $\log n$ vertices.  By using sparse regularity together with the large minimum degree result we are able to begin this process with a tree having $\Theta(n)$ leaves as opposed to the $\Theta(pn)$ leaves we would be able to guarantee without sparse regularity.

In Section \ref{ssec:examples}, we provide examples of graphs and colorings which give lower bounds on $\tc_r(G)$ (and hence $\tp_r(G)$).  In Section \ref{ssec:distinct} we consider the variant where we require that the components in the cover must be of distinct colors. In Section \ref{ssc:comneigh}, we give a simple upper bound on $\tc_r(G)$ and prove a result about graphs in which every $r+1$ vertices have a common neighbor.

Section \ref{sec:mindeg} is devoted to proving the large minimum degree versions (including complete versions) of our results.  In Section \ref{ssec:part} we prove Theorems \ref{thm:completepartition} and \ref{thm:mindegpartition}. The first provides a slight improvement on the bound in Theorem \ref{thm:HK} and the second extends the theorem to graphs with large minimum degree.  In Section \ref{ssec:cov} we prove Theorem \ref{thm:2colorcover}, which provides a tight minimum degree condition on $G$ such that $\tc_2(G)\le 2.$  In Section \ref{ssec:lmcdeg}, we prove that $r$-colored graphs with large minimum degree have large monochromatic components.  

In Section \ref{sec:srlba}, we continue with the second step of the method described above by stating the sparse regularity lemma of Kohayakawa \cite{Koh} and R\"{o}dl (see \cite{Con})  as well as collecting various lemmas which will be useful for the proof. Lemma \ref{lem:linsizecommonleaves} shows that sparse edge colored random graphs have nearly spanning ``robust'' tree partitions.   

In Section \ref{ssc:randomlem} we deduce some properties of $\gnp$ which will be used in Section \ref{sec:mainthmpf}.

In Section \ref{sec:mainthmpf} we prove Theorem \ref{thm:gnp_threshold} and \ref{thm:gnp_threshold_low}.  Theorem \ref{thm:gnp_threshold}(ii) and Theorem \ref{thm:gnp_threshold_low}(i),(ii) will follow from the results of Section \ref{ssc:randomlem}.  For Theorem \ref{thm:gnp_threshold}(i), we are able to exploit the fact that there are only two colors to improve the general result of Theorem \ref{thm:gnp_threshold}(ii).  The proof of Theorem \ref{thm:gnp_threshold}(iii) is discussed in the second paragraph of this section. Finally, in Section \ref{ssec:lmc}, we prove Theorem \ref{thm:randomlargemono}.


In Section \ref{sec:conc} we collect some conjectures and open problems. 

\subsection{Notation}
We use the following notation throughout the paper. As usual, $N(v)$ represents the neighborhood of $v$ and as we deal mainly with colored graphs, if $c$ is a color then  $N_c(v)$ represents the neighborhood of $v$ in the subgraph of $c$ colored edges and $\deg_c(v) = |N_c(v)|$.  If $S$ is a set of vertices, then  $N_c(v,S) = N_c(v)\cap S$ and $\deg_c(v,S) = |N_c(v,S)|$. 
We let $N^{\cap}(S)=\bigcap_{v\in S}N(v)$ and $N^\cup(S)=\bigcup_{v\in S}N(v)$.
For two sets of vertices $X$ and $Y$, $e(X,Y)$ represents the number of edges with one endpoint in $X$ and the other in $Y$. 
For two sequences $a_n, b_n$, we write $a_n =o(b_n)$  if $a_n/b_n \to 0$ as $n\to \infty$ and $a_n = \omega(b_n)$ if $a_n /b_n \to \infty$  as $n\to \infty.$  
For constants $a$ and $b$, we write $a \ll b$ to mean that given $b$, we can choose $a$ small enough so that $a$ satisfies all of necessary conditions throughout the proof.  More formally, we say that a statement holds for $a\ll b$ if there is a function $f$ such that it holds for every $b$ and every $a\le f(b)$ 
In order to simplify the presentation, we will not determine these functions explicitly.  We will ignore floors and ceilings when they are not crucial to the calculation.  Logarithms are assumed to be base $e$ unless otherwise noted.

\section{Examples and observations}\label{sec:examples}

\subsection{Lower bounds on the tree cover number}\label{ssec:examples}

In this section we provide examples which give lower bounds on $\tc_r(G)$ in various settings considered throughout the paper.  We remind the reader that for all $r\geq 1$ and all graphs $G$, $\tp_r(G)\geq \tc_r(G)$.

\begin{observation}\label{obs:alphar}
Let $r\geq 1$.  For all graphs $G$, if $\alpha(G)\geq r$, then $\tc_r(G)\geq r$.  In particular, there exists a graph $G$ with $\delta(G)\geq n-r$ with $\tc_r(G)\geq r$.
\end{observation}

\begin{proof}
Choose an independent set $\{x_1,\dots, x_r\}$ and color every edge incident to $x_i$ with color $i$, then color the remaining edges arbitrarily.  None of the vertices $x_1,\dots, x_r$ are in a tree of the same color.
\end{proof}

\begin{observation}\label{obs:need_s_trees}
Let $G=(V, E)$ be a graph.  For all $1\leq r\leq s$, if $G$ contains an independent set $X$ of size $s$ such that every vertex in $V\setminus X$ has at most $r-1$ neighbors in $X$ and $X$ is not a dominating set, 
then $\tc_r(G)>s$.
\end{observation}

\begin{proof}
Start by coloring every edge in $G[V\setminus X]$ with color $r$.  The only edges not yet colored are those going between $V\setminus X$ and $X$.  Since each $v\in V\setminus X$ is incident with at most $r-1$ such edges, we can assign colors so that no vertex in $V\setminus X$ is incident with more than one edge of color $i$ for any $i\in [r-1]$.

To see that $G$ cannot be covered with $s$ monochromatic trees, note that for any pair $x,x'\in X$, $x$ and $x'$ must be in different trees; this follows since $x$ and $x'$ are not incident with any edges of color $r$ and $x$ and $x'$ have no neighbors of the same color (by the way colors were assigned to edges from $V\setminus X$ to $X$), so there are no monochromatic paths from $x$ to $x'$.  Furthermore, since $X$ is not a dominating set, there must exist at least one tree of color $r$ (since every edge in $G[V\setminus X]$ has color $r$).  This implies that $tc_r(G)\geq s+1$.  
\end{proof}

\begin{example}
\label{obs:tcrlb}
For $r\geq 1$ and $n\geq 2r+2$, there exists a graph $G$ on $n$ vertices with $\delta(G)=\ceiling{\frac{r(n-r-1)+1}{r+1}}-1$ such that $\tc_r(G)>r$.
\end{example}

\begin{proof}
There exist unique integers $m$ and $q$ such that $n=(r+1)m+q$ with $0\leq q\leq r$.  Set aside vertices $u_1,\dots, u_{r+1}$ and then equitably partition the remaining $n-(r+1)$ vertices into sets $V_1,\dots, V_{r+1}$; that is, partition the remaining vertices into sets $V_1,\dots, V_{r+1}$ so that $|V_1|=\dots=|V_{r+1-q}|=m-1$ and if $q\geq 1$, $|V_{r+1-q+1}|=\dots=|V_{r+1}|=m$.

Now add the following colored edges: 
\begin{itemize}
\item $u_i\sim V_{j}$ in color $i$ for $1\le i< j \le r+1$ 
\item $u_i\sim V_{j}$ in color $i-1$ for $1\le j< i\le r+1 $
\item $V_i \sim V_j$ in color $r$ for $1\le i<j \le r$
\item $V_{r+1} \sim V_i$ in color $1$ for $2\le i\le r$
\item $V_i\sim V_i$ with arbitrary colors for all $i$
\end{itemize}
Note that if $i\neq j$, then $u_i$ and $u_j$ cannot be in the same monochromatic component. Among $u_1,\ldots, u_{r+1}$, color $i$ only appears incident to $u_i$ and $u_{i+1}$ but their neighborhoods in color $i$ are disjoint; these neighborhoods remain disjoint in color $i$ even when the edges between the $V_i$'s are considered. 

By construction, it is clear that $u_{r+1}$ has the smallest degree of the $u_i$ and vertices in $V_1$ have the smallest degree of the vertices in the $V_i$. Note that $\deg(u_{r+1}) = |V_1|+\cdots+|V_r|$ which is $r(m-1)$ if $q =0$ and $r(m-1)+q-1$ if $q=1,\ldots, r$. Now since $n=(r+1)m+q$, we have 
\[\ceiling{\frac{r(n-r-1)+1}{r+1}}-1 = r(m-1)+ \cefrac{rq+1}{r+1} -1.\]
Since $q \ge \frac{rq+1}{r+1} >q-1$ for all $1\le q \le r$ and $\cefrac{1}{r+1}=1$, $u_{r+1}$ satisfies the claimed degree condition for all $0\le q\le r$. If $v_1\in V_1$, then $\deg(v_1) \ge (r-1)(m-1) + (m-2) + r = r(m-1) + r-1 \ge \deg(u_{r+1})$. 
\begin{figure}[!htb]
    \centering
    \begin{minipage}{.5\textwidth}
        \centering
        \includegraphics[scale=1]{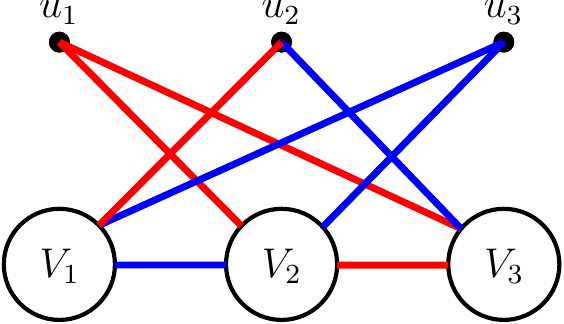}

    \end{minipage}%
    \begin{minipage}{0.5\textwidth}
        \centering
        \includegraphics[scale=.7]{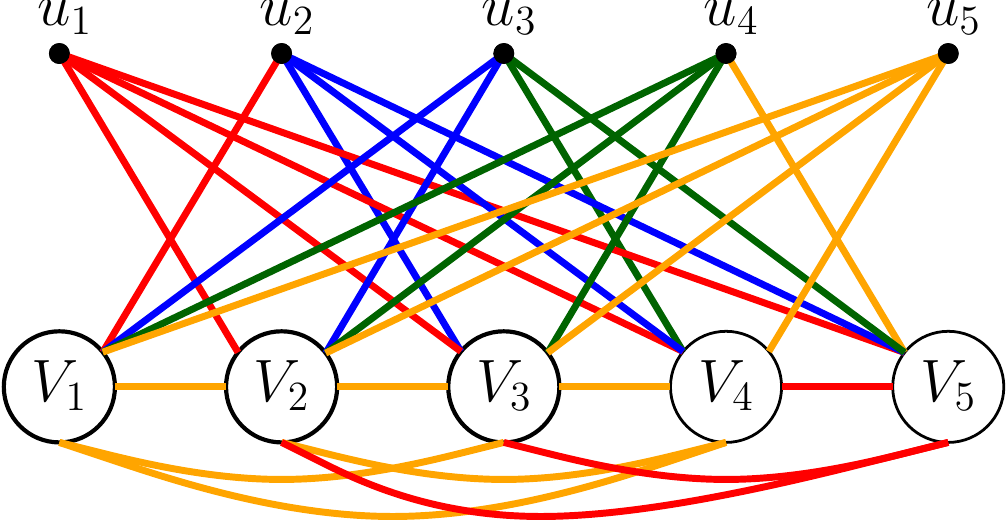}
    \end{minipage}
     \caption{Graphs with $\tc_2(G)>2$ and $\tc_4(G)>4$ respectively.}
\end{figure}

\end{proof}

\subsection{Covering with trees of distinct colors}\label{ssec:distinct}

\begin{definition} Let $G$ be a (multi)graph.
 Say $G$ has property $\cT\cP_r\, (\cT\cC_r)$ if in every $r$-coloring of the edges of $G$ there is a \emph{partition} \emph{(cover)} of $V(G)$ with at most $r$ trees of \emph{distinct} colors.
\end{definition}

Next we provide examples of graphs which cannot be covered by $r$ trees of distinct colors. We note however, that these graphs can be partitioned into just two components of the same color.

\begin{example}\label{ex:tcrG}
For all $r\geq 1$ and $n\geq 2^r$, there exists a graph $G$ on $n$ vertices with $\delta(G)= \floor{(1-\frac{1}{2^r})n}-1$ such that $G$ does not have property $\cT\cC_r$.
\end{example}

\begin{proof}
We construct a graph $G$ with $n = m\cdot 2^r$ vertices.  The vertices are partitioned into $2^r$ sets of size $m$.  We index these sets by binary strings of length $r$. So $V = \dot{\bigcup}_{\tbf{b}\in\set{0,1}^r}V_{\tbf{b}}$. Every vertex within such a set will also be referred to by the index of its set. For a binary string $\tbf{b}$, let $\neg\tbf{b}$ represent the string with all bits flipped.

Include every edge between vertices which agree on at least one index. So $V_{\tbf{b}}\sim V_{\tbf{b}'}$ iff $\tbf{b}\neq \neg\tbf{b}'$ (all edges within the $V_\tbf{b}$ are present as well). This graph has $\delta(G) = n-1-m = (1-\frac{1}{2^r})n -1$. Now color each edge with the smallest coordinate on which the endpoints indices agree. For example, two vertices with indices $(0,1,0,0)$ and $(1,0,0,1)$, would be connected by an edge of color 3.

We claim that $G$ cannot be covered by $r$ monochromatic components of distinct colors.  Any connected subgraph of color $i$ can only contain vertices which agree on the $i$th coordinate. Suppose the component of color $i$ only covers vertices with $b_i$ in the $i$th component. Then the vertices with index $\neg(b_1,\ldots, b_r)$ are not covered by any of the components.
When $n = m\cdot 2^r +q$  with $q < 2^r$, we proceed in the same way, but partition the vertices into $q$ sets of size $m+1$ and $2^r-q$ sets of size $m$. 

\end{proof}

\subsection{Simple upper bounds on the tree cover number}\label{ssc:comneigh}

\begin{observation}\label{rysertrivial}
Let $r\geq 1$. For all (multi)graphs $G$, $\tc_r(G)\leq r\alpha(G)$.
\end{observation}

\begin{proof}
Let $a:=\alpha(G)$ and let $X=\{x_1, \dots, x_a\}$ be a maximum independent set.  So every vertex in $V(G)\setminus X$ has a neighbor in $X$.  Taking the stars centered at $x_1,\dots, x_a$ gives a collection of at most $r\alpha(G)$ monochromatic components which cover $V(G)$.
\end{proof}

\begin{proposition}\label{prop:comneigh}
Let $r\geq 2$ and let $G$ be a graph having the property that every set of $r+1$ vertices have a common neighbor, then $\tc_r(G)\leq r^2$.
\end{proposition}

\begin{proof}
Consider any $r$-coloring of $G$ and let $H$ be an $r$-colored auxiliary (multi)graph on $V(G)$ where $uv\in E(H)$ of color $i$ if and only if there is a path in $G$ of color $i$ from $u$ to $v$.  Since every set of $r+1$ vertices of $G$ have a common neighbor and there are at most $r$ colors, this implies $\alpha(H)\leq r$.  Thus by Observation \ref{rysertrivial}, we have $\tc_r(H)\leq r\alpha(H)\leq r^2$.  Note that a monochromatic component in $H$ corresponds to a monochromatic component in $G$ giving the result.
\end{proof}

\section{Monochromatic trees in graphs with large minimum degree} \label{sec:mindeg}

\subsection{Partitions}
\label{ssec:part}

We start by proving a lemma which we will use in the proofs of Theorems \ref{thm:completepartition}, \ref{thm:mindegpartition}, and  \ref{thm:gnp_threshold}(iii). 

\begin{lemma}\label{lem:Ypartition}
Let $k\geq 2$.  If $G$ is an $Y,Z$-bipartite graph such that for all $v\in Z$, $\deg(v, Y)>k\log|Z|$, then in every $k$-coloring of the edges of $G$, there exists a partition $\{Y_1, \dots, Y_k\}$ of $Y$ such that for all $v\in Z$, there exists $i\in [k]$ such that $N_i(v)\cap Y_i\neq \emptyset$.
\end{lemma}

\begin{proof}
Randomly color the vertices of $Y$ with colors from $[k]$, giving us a partition $\{Y_1,\dots, Y_k\}$ of $Y$ (with possibly empty parts).  The probability that some vertex $v\in Z$ has $N_i(v)\cap Y_i=\emptyset$ for all $i\in [k]$ is 
$$(1-\frac{1}{k})^{\deg(v, Y)}<(1-\frac{1}{k})^{k\log|Z|}\leq e^{-k\log|Z|/k}=\frac{1}{|Z|}.$$ 
So by the union bound the probability of at least one failure is less than 1, and thus there exists a partition of $Y$ with the desired property.
\end{proof}


We now prove Theorem \ref{thm:completepartition} which says that $\tp_r(K_n)\leq r$ holds provided $n\geq 3r^2r!\log r$ (improving the lower bound of $n\geq \frac{3r^4r!\log r}{(1-1/r)^{3(r-1)}}$ from Theorem \ref{thm:HK})
and illustrates the idea for both Theorem \ref{thm:mindegpartition} and Theorem \ref{thm:gnp_threshold}(iii).  We note that our proof follows a similar procedure as the proof in \cite{HK}, except that at the end of the process we use Lemma \ref{lem:Ypartition} instead of a greedy algorithm.

\begin{proof}[Proof of Theorem \ref{thm:completepartition}]
\noindent
\textbf{Step 1:} Let $x_1\in V(G)$ and let $Y_1$ be the largest monochromatic neighborhood of $x_1$, say the color is $1$.  Note that $|Y_1|\geq (n-1)/r$.  If every vertex in $V\setminus Y_1$ has a neighbor of color $1$ in $Y_1$, then stop as we would already have the desired tree partition.  So some vertex $x_2$ has at least $\frac{1}{r-1}|Y_1|$ neighbors of say color $2$ in $Y_1$.  Set $Y_2:=Y_1\cap N_2(x_2)$.  

For $2\leq i\leq r-1$, assuming $Y_i$ has already been defined, we do the following: if for all $v\in V\setminus Y_{i}$, $|(\bigcup_{j=1}^iN_j(v))\cap Y_i|>i\log n,$ 
then set $k:=i$, $Y:=Y_k$, and $Z=V\setminus (\{x_1,\dots, x_k\}\cup Y_k)$ then proceed to Step 2. Otherwise some vertex $x_{i+1}\in V\setminus Y_{i}$ has at most $i\log n$ neighbors having colors from $[i]$ in $Y_{i}$ and thus $x_{i+1}$ has at least $\frac{1}{r-i}(|Y_{i}|-i\log n)$ neighbors of color say $i+1$, in $Y_{i}$.  Set $Y_{i+1}:=Y_{i}\cap N_{i+1}(x_{i+1})$.  Continue in this manner until we go to Step 2 or until $Y_r$ has been defined.  After we complete the $i=(r-1)$-th step, we have  
$$|Y_r|\geq \frac{n-1}{r!}-\log n\sum_{j=1}^{r-2}\frac{r-j}{j!}\geq r\log n$$
where the last inequality holds provided $\frac{n}{\log n}\geq r!\sum_{j=0}^{r-2}\frac{r-j}{j!}$.  Recall that $n\geq 3r^2r!\log r$ and note that $\sum_{j=0}^{r-2}\frac{r-j}{j!}\leq er$.  For $r\geq 6$, we have $\frac{n}{\log n}\geq err!\geq r!\sum_{j=0}^{r-2}\frac{r-j}{j!}$, and for $2\leq r\leq 5$ we directly verify $\frac{n}{\log n}\geq r!\sum_{j=0}^{r-2}\frac{r-j}{j!}$.  Now set $k:=r$, $Y:=Y_k$, and $Z=V\setminus (\{x_1,\dots, x_r\}\cup Y_r)$ then proceed to Step 2.

\noindent
\textbf{Step 2:}
Note that for all $v\in Z$, $|(\bigcup_{j=1}^kN_j(v))\cap Y_k|\geq k\log n>k\log|Z|$.  Thus we may apply Lemma \ref{lem:Ypartition} to get a partition $\{Y_1, \dots, Y_k\}$ of $Y$ such that for all $v\in Z$, there exists $i\in [k]$ such that $N_i(v)\cap Y_i\neq \emptyset$.  
Let each $v\in Z$ choose an arbitrary such $i$ and an arbitrary neighbor in $N_i(v) \cap Y_i$. Then $x_i$ along with $Y_i$ and all the $v\in Z$ which chose neighbors in $Y_i$ form a tree of color $i$ and radius at most $2$. Thus we have a partition into $k\leq r$ monochromatic trees.

\end{proof}

A key element in the preceding proof was to first build a monochromatic tree cover in which the common intersection of all of the trees was a large enough set of leaves.  We now explicitly define this structure.

\begin{definition}
A \tbf{$(k,l,n)$-absorbing tree partition} is a collection of trees $T_1,\dots, T_k$ together with a \tbf{common leaf set} $L$ of size $l$ such that
\begin{enumerate} 
\item $|\bigcup_{i\in [k]}V(T_i)|=n$,
\item the edges of $T_i$ have color $i$ for all $i\in [k]$, 
\item every vertex in $L$ is a leaf of $T_i$ for all $i \in [k]$, and \item for all $i\neq j$, $V(T_i)\cap V(T_j)=L$.
\end{enumerate}
\end{definition}

Note that if every $r$-coloring of a graph $G$ on $n$ vertices contains a $(k,l,n)$-absorbing tree partition for some $1\leq k\leq r$ and $l\geq 0$, then by arbitrarily assigning the leaves to the trees we have $\tp_r(G)\le r$ (with trees of distinct colors).

We will consider absorbing tree partitions in two different settings: first, in Theorem \ref{thm:mindegpartition} we wish to optimize the bound on the minimum degree so that $\tp_r(G)\leq r$, and second, we will apply Theorem \ref{thm:mindegpartition} in a setting where the graph is nearly complete, in which case we do not need so much control over the minimum degree as we need control over the size of the common leaf set.  So for the purposes of streamlining, we combine everything we want into the following statement, which has a parameter $\ep$ related to the minimum degree and a parameter $\alpha$ which is related to the size of the leaf set and the lower bound on $n$.  The method of proof will be similar to that of Theorem \ref{thm:completepartition}; however, the calculations are different as here we are attempting to optimize the minimum degree instead of the lower bound on $n$.

\begin{theorem}\label{thm:mindegpartition}
Let $r\geq 2$, $0<\ep<\frac{1}{er!}$, $\alpha=\frac{1}{er!}-\ep$, and $n_0=\max\{\frac{12}{\alpha^2}\log(\frac{6}{\alpha^2}), \frac{4r}{\alpha}\log(\frac{2r}{\alpha})\}$.  If $G$ is a graph on $n\geq n_0$ vertices with $\delta(G)\geq (1-\ep)n$, then in every $r$-coloring of $G$ there either exists a $(1,l,n)$-absorbing tree partition with $l\geq n-\frac{2}{\alpha}\log n$ (i.e. a monochromatic spanning tree with at least $n-\frac{2}{\alpha}\log n$ leaves) or a $(k,l,n)$-absorbing tree partition with $2\leq k\leq r$ and $l\geq \alpha n/2$.
\end{theorem}

We will need the following two statements in the proof of Theorem \ref{thm:mindegpartition}.

\begin{observation}\label{logn/n}
Let $x\in \mathbb{R}$ with $x\geq 2$.  If $n\geq 2x\log x$, then $\frac{\log n}{n}<\frac{1}{x}$.  
\end{observation}

\begin{proof}
We first note that $\frac{\log n}{n}$ is strictly decreasing since $n\geq 2x\log x>e$.  Now since $2x\log x<x^2$, we have $\frac{\log(2x\log x)}{2x\log x}=\frac{\log(2x\log x)}{x\log x^2}<\frac{1}{x}$.
\end{proof}

\begin{lemma}\label{lem:manyleaves}
Let $0<\alpha\leq 1$ and $n_0=\frac{12}{\alpha^2}\log(\frac{6}{\alpha^2})$ and let $G$ be a graph on $n\geq n_0$ vertices.  If there exists $x\in V(G)$ such that for all $v\in V(G)$, $\deg(v, N(x))\geq \alpha n$, then $G$ has a spanning tree with at least $n-\frac{2}{\alpha}\log n$ leaves.
\end{lemma}

\begin{proof}
We will show that $x$ along with at most $\frac{2}{\alpha}\log n$ of its neighbors form a dominating set.
Set $Y_1:=N(x)$ and $Z_1=V\setminus (\{x\}\cup Y_1)$.  For $1\leq i \leq \ceiling{\frac{3}{2\alpha}\log n}-1$,  
do the following: If $Z_i\neq \emptyset$, let $y_{i}$ be the vertex in $Y_i$ with the largest degree to $Z_i$. Set $Y_{i+1}=Y_i\setminus \{y_{i}\}$ and $Z_{i+1}=Z_i\setminus N(y_{i})$.  Since $\deg(y_i, Z_i)\geq \frac{\alpha n-i}{|Y_1|-i}|Z_i|> \frac{2\alpha}{3} |Z_i|$ (where the second inequality holds by Observation \ref{logn/n} and the bound on $n$) and thus $$|Z_{i+1}|< (1-2\alpha/3)|Z_i|\leq (1-2\alpha/3)^{i}|Z_1|\leq (1-2\alpha/3)^{i+1}n\leq 1$$ when $i+1\geq \frac{3}{2\alpha}\log n$.  Thus when the process stops, we have a spanning tree with at most $\ceiling{\frac{3}{2\alpha}\log n}\leq \frac{2}{\alpha}\log n$ non-leaves.
\end{proof}

\begin{proof}[Proof of Theorem \ref{thm:mindegpartition}]
\noindent
\textbf{Step 1:} Let $x_1\in V(G)$ and let $Y_1$ be the largest monochromatic neighborhood of $x_1$, say the color is $1$.  Note that $|Y_1|\geq \frac{1}{r}(1-\ep)n$.  If for all $v\in V\setminus Y_1$, $\deg_1(v, Y_1)\geq \alpha n$, then since $n\geq \frac{12}{\alpha^2}\log(\frac{6}{\alpha^2})$, we may apply Lemma \ref{lem:manyleaves} to get a  monochromatic spanning tree (in color 1) with at least $n- \frac{2}{\alpha}\log n$ leaves and we are done.  Otherwise some vertex $x_2$ has at least $\frac{1}{r-1}(|Y_1|-\ep n-\alpha n)$ neighbors of say color $2$ in $Y_1$.  Set $Y_2:=Y_1\cap N_2(x_2)$.  

For $i\geq 2$, do the following: if for all $v\in V\setminus Y_{i}$, $|(\bigcup_{j=1}^iN_j(v))\cap Y_i|\geq \alpha n$,
then set $k:=i$, $Y:=Y_k$, and $Z=V\setminus (\{x_1,\dots, x_k\}\cup Y_k)$ then proceed to Step 2. Otherwise some vertex $x_{i+1}\in V\setminus Y_{i}$ has at least $\frac{1}{r-i}(|Y_{i}|-\ep n-\alpha n)$ neighbors of color say $i+1$, in $Y_{i}$.  Set $Y_{i+1}:=Y_{i}\cap N_{i+1}(x_{i+1})$.  Continue in this manner, until we go to Step 2 or until $i=r$.  If $i=r$, then 
$$|Y_r|\geq \left(\frac{1}{r!}-\ep\sum_{j=1}^{r}\frac{1}{j!}-\alpha\sum_{j=1}^{r-1}\frac{1}{j!}\right)n\geq \left(\frac{1}{r!}-\ep (e-1)-\alpha(e-1)\right)n$$ 
and thus every vertex in $V\setminus Y_r$ has at least 
$$|Y_r|-\ep n\geq \left(\frac{1}{r!}-\ep (e-1)-\alpha(e-1)\right)n-\ep n=\alpha n$$ 
neighbors in $Y_r$.  Now set $k:=r$, $Y:=Y_k$, and $Z=V\setminus (\{x_1,\dots, x_k\}\cup Y_k)$ then proceed to Step 2.

\noindent
\textbf{Step 2:} First set aside $\alpha n/2$ vertices from $Y$ to be the common leaf set of the absorbing tree partition and let $Y'$ be the remaining vertices in $Y$.  Every vertex in $Z$ still has at least $\alpha n/2$ neighbors in $Y'$.  Since $n\geq \frac{4r}{\alpha}\log(\frac{2r}{\alpha})$, Observation \ref{logn/n} implies that $\alpha n/2>r\log n$ and thus we can apply Lemma \ref{lem:Ypartition} to get a partition $\{Y_1', \dots, Y_k'\}$ of $Y'$ such that for all $v\in Z$, there exists $i\in [k]$ such that $N_i(v)\cap Y_i'\neq \emptyset$.  By arbitrarily choosing such a $Y_i'$ for each $v\in Z$, we have a $(k,l,n)$-absorbing tree partition with $2\leq k\leq r$ and $l\geq \alpha n/2$.
\end{proof}

While we are not able to prove that the bound on the minimum degree in Theorem \ref{thm:mindegpartition} is optimal, Observation \ref{obs:alphar} shows that there are graphs with $\delta(G)\geq n-r$ for which $\tp_r(G)\geq \tc_r(G)\geq r$, and thus the goal in the minimum degree version of the problem (optimizing $\delta(G)$ while maintaining $\tp_r(G)\leq r$) is different from the goal in the case of complete graphs (proving $\tp_r(K_n)\leq r-1$).  

In Theorem \ref{thm:mindegpartition} we actually prove that the trees have distinct colors, so it is natural to ask the question of how the minimum degree threshold for partitioning (covering) into $r$ trees compares to the minimum degree threshold for partitioning (covering) into $r$ trees of distinct colors (see Section \ref{ssec:distinct}).

\subsection{Covers}
\label{ssec:cov}

For the cover version of the $r=2$ case, we can actually prove a tight bound on the minimum degree (see Example \ref{obs:tcrlb}).

\begin{theorem}\label{thm:2colorcover}
Let $n\geq 1$.  For all graphs $G$ on $n$ vertices, if $\delta(G)\geq \frac{2n-5}{3}$, then $\tc_2(G)\leq 2$.  
\end{theorem}

\begin{proof}
Suppose that $n=3m+q$ where $q\in\set{0,1,2}$. Then $\delta(G)\ge \frac{2n-5}{3}$ translates to $\delta(G)\ge 2m-1+\flfrac{q}{2}$.

Suppose $G$ is $2$-colored and let $\cT=\{R_1,\dots, R_k$, $B_1,\dots, B_l\}$ be a monochromatic component cover of $G$ with the fewest number of components, where each component is maximal; and with respect to this, choose $\cT$ so that as many different colors are represented as possible.  Without loss of generality suppose $k\geq l$.  We are done unless $|\cT|\geq 3$.  It is clear from minimality of the number of components, that for each component $T\in \cT$, there is a non-empty subset of vertices $\phi(T)$ such that every vertex in $\phi(T)$ is not contained in any other component $S\in \cT\setminus \{T\}$.

\textbf{Case 1} (There is at least one component of each color). 

Since $k\geq l$ and $k+l\geq 3$, we have $k\geq 2$.  Suppose first that there exist vertices $u_i\in \phi(R_i)$, $u_j\in \phi(R_j)$, and $v_h\in \phi(B_h)$ such that $u_iu_j\not\in E(G)$.  By the maximality of the components, $u_iv_h, u_jv_h\not\in E(G)$.  So 
$$|N(u_i)\cap N(u_j)\cap N(v_h)|\geq n-3-3(n-3-\delta(G))=3\delta(G)-2n+6\geq 1$$
So let $w\in N(u_i)\cap N(u_j)\cap N(v_h)$.  If $w$ is in a blue component $B'\in \cT$, then $w$ cannot be adjacent to $u_i$ or $u_j$ via a blue edge (as this would imply that $u_i$ or $u_j$ is contained in $B'$).  So $w$ is adjacent to $u_i$ and $u_j$ via red edges, but this contradicts the fact that $u_i$ and $u_j$ are in different red components.  So suppose that for all blue components $B'\in \cT$, $w\not\in V(B')$.  This implies that $wv_h$ must be a red edge and that $w$ is in a red component of $\cT$, but again this contradicts the fact that $v_h$ is not contained in any red component of $\cT$.  In either case, we get a contradiction. 

So we may assume that the vertices of $\phi(R_1),\dots, \phi(R_k)$ induce a complete $k$-partite blue graph $B'$.  However, this implies that $\{B_1, \dots, B_l, B'\}$ is a cover with fewer components.

\textbf{Case 2} (All of the components have the same color).  

Since $k\geq l$, $\cT=\{R_1,\dots, R_k\}$. Without loss of generality suppose $|R_1|\leq |R_2|\leq \dots\leq |R_k|$.  Note that since $k\geq 3$ and all components are red, we have $2\leq |R_1|\leq m$.  Since $R_1$ is maximal, every edge leaving $R_1$ is blue and thus for all $v\in R_1$ we have
\begin{equation}\label{eq:R1blue}
|N_B(v)\cap (V(G)\setminus R_1)|\geq \delta(G)-(|R_1|-1)\geq 2m+\floor{\frac{q}{2}}-|R_1|\geq m+\floor{\frac{q}{2}}.
\end{equation}

From \eqref{eq:R1blue}, and the fact that $|V(G)\setminus R_1|\leq 3m+q-2$, we see that for any set of three vertices $\{x,y,z\}$ in $R_1$, some pair of $\{x,y,z\}$ must have a common blue neighbor in $V(G)\setminus R_1$. This implies that there are either one or two blue components which cover the vertices of $R_1$. 
If there were only one, we would be in Case 1. So assume that there are two blue components $B_1$ and $B_2$ which cover every vertex in $R_1$.  Now using \eqref{eq:R1blue}, we get that 
$$|B_1|+|B_2|\geq |R_1|+2(2m+\floor{\frac{q}{2}}-|R_1|)=4m+2\floor{\frac{q}{2}}-|R_1|\geq 3m+2\floor{\frac{q}{2}}\geq 3m+q-1.$$
So either $B_1$ and $B_2$ form a cover with two components, or $B_1$, $B_2$, and the red component containing the leftover vertex form a cover with two blue components and a red component and thus we are in Case 1.


\end{proof}

\subsection{Large monochromatic components}
\label{ssec:lmcdeg}

We use the following lemma of Liu, Morris, and Prince \cite{LMP} (an essentially equivalent version of this lemma was independently proved by Mubayi \cite{Mub}).  A \tbf{double-star} is a tree having at most two vertices which are not leaves.  

\begin{lemma}[Lemma 9 in \cite{LMP}]\label{lem:LMP}
Let $c\geq 0$ and let $G$ be an $X,Y$-bipartite graph on $n$ vertices.  If $e(G)\geq c|X||Y|$, then $G$ has a double-star of order at least $cn$.
\end{lemma}

\begin{theorem}\label{thm:mindeglargemono}
Let $r\geq 2$, and let $0<\ep\leq \frac{1}{2}$.  
If $G$ is a graph on $n$ vertices with $\delta(G)\geq (1-\ep)n$,
then every $r$-coloring of $G$ contains either a monochromatic tree of order at least $(1-\ep)n$ or a monochromatic double-star of order at least $(1-2\ep)\frac{n}{r-1}$.
\end{theorem}

\begin{proof}
Let $H$ be the largest monochromatic tree in $G$, say of color $r$ and suppose that $|V(H)|<(1-\ep)n$.  Set $X=V(H)$ and $Y=V(G)\setminus X$ and let $B=G[X,Y]$ be the bipartite graph induced by the bipartition $\{X,Y\}$.  Without loss of generality suppose $|X|\geq |Y|$ and note that $e(B)\geq |Y|(|X|-\ep n)\geq (1-2\ep)|X||Y|$. Note that by the maximality of $H$, there are no edges of color $r$ in $B$, so at least $(1-2\ep)\frac{1}{r-1}|X||Y|$ of the edges of $B$ are say color 1.  So by Lemma \ref{lem:LMP}, we have a monochromatic double-star with at least $(1-2\ep)\frac{n}{r-1}$ vertices. 

\end{proof}

\section{Sparse Regularity and Basic Applications}
\label{sec:srlba}

We will use of a variant of Szemer\'edi's regularity lemma \cite{Szem} for sparse graphs, which was proved independently by Kohayakawa \cite{Koh} and R\"odl (see \cite{Con}) and later generalized by Scott \cite{Scott}.

Say that a $U,V$-bipartite graph is \tbf{weakly-$(\ep, q)$-regular} if for all $U'\subseteq U$ and $V'\subseteq V$ with $|U'|\geq \ep |U|$ and $|V'|\geq \ep |V|$, $e(U', V')\geq q$.  
Given disjoint sets $X,Y$ and $0<p\leq 1$, define the \tbf{$p$-density} of $(X,Y)$, denoted $d_p(X,Y)$, by $$d_p(X,Y) = \frac{e(X,Y)}{p|X||Y|}.$$  We say the bipartite graph $G[U,V]$ induced by disjoint sets $U$ and $V$ is \tbf{$(\ep,p)$-regular} or that $(U,V)$ is an \tbf{$(\ep,p)$-regular pair} if for all subsets $U' \subseteq U$, $V'\subseteq V$ with $|U'| \ge \ep|U|$,  $|V'| \ge \ep|V|$ we have $|d_p(U',V') - d_p(U,V)| \le \ep$. Given a graph $G=(V,E)$, a partition $\{V_1, \ldots, V_k\}$ of $V$ is said to be an \tbf{$(\ep, p)$-regular partition} if it is an equitable partition (i.e. $| |V_i| - |V_j|  | \le 1$ for all $i,j\in [k]$) and all but at most $\ep \binom{k}{2}$ pairs $(V_i, V_j)$ induce $(\ep, p)$-regular pairs.

%

We will use the following $r$-colored version of the sparse regularity lemma due to Scott \cite[Theorem 4.1]{Scott}.
 
\begin{lemma}\label{lem:ssrlc}
For every $\ep>0$ and $m,r \geq 1$, there exists $M$ such that if $G_1,\ldots, G_r$ are edge-disjoint graphs on vertex set $V$ with $|V|\geq m$ and where $p_i=\frac{e(G_i)}{\binom{|V|}{2}}>0$ is the density of $G_i$ for all $i\in [r]$, there exists a partition $\{V_1, \dots, V_k\}$ of $V$ with $m \le k \le M$ such that for all $i\in [r]$, $\{V_1, \dots, V_k\}$ is an $(\ep, p_i)$-regular partition of $G_i$. 
\end{lemma}

Let $G$ be an $r$-colored graph on $n$ vertices with $p=\frac{e(G)}{\binom{n}{2}}$ and $r$-coloring $\{G_1, \dots, G_r\}$, where for all $i\in [r]$, $p_i=\frac{e(G_i)}{\binom{n}{2}}$; note that $p=\sum_{i=1}^rp_i$.  We define the \tbf{$(\ep, p, \delta)$-reduced graph} $\Gamma$ as follows:  Let $\{V_1, \dots, V_k\}$ be the partition of $G$ obtained from an application of Lemma \ref{lem:ssrlc}.  Let $V(\Gamma)=\{V_1, \dots, V_k\}$ and say that $\{V_i, V_j\}$ is a $c$-colored edge of $\Gamma$ if the $p_c$-density of $(V_i, V_j)$ in $G_c$ is at least $\delta$.  Note that $\Gamma$ is a (possibly) multicolored graph.

The following simple lemma is typically applied to the reduced graph obtained after an application of Lemma \ref{lem:ssrlc}.

\begin{lemma}\label{lem:nicereduced}
Let $\alpha>0$.  If $H$ is a graph on $k\geq \frac{2}{\sqrt{\alpha}}$ vertices with $e(H)\geq (1-\alpha)\binom{k}{2}$, then there exists a subgraph $H'\subseteq H$ such that $|V(H')|\geq (1-\sqrt{\alpha})k$ and $\delta(H')\geq (1-2\sqrt{\alpha})k$.  
\end{lemma}

\begin{proof}
As there are at most $\alpha\binom{k}{2}\leq \alpha\frac{k^2}{2}$ non-edges, there are at most $\sqrt{\alpha}k$ vertices $V'$ which have at least $\frac{\sqrt{\alpha}}{2}k$ non-neighbors.  Let $H'=H[V']$.  We have $|V(H')|\geq k-\sqrt{\alpha}k= (1-\sqrt{\alpha})k$ and $\delta(H')\geq k-1-\frac{\sqrt{\alpha}}{2}k-\sqrt{\alpha}k\geq (1-2\sqrt{\alpha})k$.
\end{proof}

We will apply Lemma \ref{lem:ssrlc} to an $r$-colored graph $G\sim \gnp$ with $p=\frac{\omega(1)}{n}$ to get a partition $\{V_1, \dots, V_k\}$.  We will only require the very weak condition that each pair which is $(\ep, p_i)$-regular for all $i\in [r]$ is weakly-$(\ep, 1)$-regular in some color $c\in [r]$.  This will follow as a simple consequence of the union bound and the Chernoff bound, which we state first for convenience (see e.g.~Corollary 21.7 in \cite{FK}).  

\begin{theorem}[Chernoff]Let $X$ be a binomially distributed random variable.  Then for $\alpha>0$, 
\[\Pr{X\leq (1-\alpha)\E{X}}\leq \exp(-\alpha^2\E{X}/2).\]
\end{theorem}

\begin{lemma}\label{lem:edgedist}
For all $\eta>0$, there exists $C$ such that if $p\geq \frac{C}{n}$, then a.a.s.~$G\sim \gnp $ has the property that for all $X, Y\subseteq V(G)$ with $X\cap Y=\emptyset$ and $|X|, |Y|\geq \eta n$, $$3p|X||Y|/2\geq e(X,Y)\geq p|X||Y|/2.$$
\end{lemma}

\begin{lemma}\label{lem:weakreginsomec}
Suppose we have applied Lemma \ref{lem:ssrlc} with $0<\ep<\frac{1}{2r}$ to an $r$-colored graph $G\sim \gnp$ with $p=\frac{\omega(1)}{n}$ to get a partition $\{V_1, \dots, V_k\}$.  Then every pair $(V_i,V_j)$ which is $(\eps, p_l)$-regular for all $l\in[r]$ is weakly-$(\eps, 1)$-regular in some color $c\in[r]$.
\end{lemma}

\begin{proof}
Suppose the pair $(V_i, V_j)$ is $(\ep, p_l)$-regular for all $l\in [r]$.  Lemma \ref{lem:edgedist} implies that $e(V_i, V_j)\geq p|V_i||V_j|/2$ and thus for some color $c\in [r]$, we have $e_c(V_i, V_j)\geq p|V_i||V_j|/2r$.  So if $V_i'\subseteq V_i$ and $V_j'\subseteq V_j$ with $|V_i'|, |V_j'|\geq \ep\frac{n}{k}$, then by $(\eps, p_c)$-regularity, we have
\[\frac{e_c(V_i', V_j')}{p_c|V_i'||V_j'|} = d_{p_c}(V_i',V_j') \ge d_{p_c}(V_i, V_j) - \eps = \frac{e_c(V_i, V_j)}{p_c|V_i||V_j|} - \eps \ge \frac{p}{2rp_c} - \eps\]
and so 
\[e_c(V_i', V_j')\geq (\frac{p}{2r}-\ep p_c)|V_i'||V_j'|>0,\]
and thus $(V_i, V_j)$ is weakly-$(\ep, 1)$-regular in color $c$.  
\end{proof}

\subsection{Nearly spanning tree partitions}

\begin{lemma}\label{lem:leafytree}
Let $0 < \eps < 1/3$ and let $G$ be a $U,V$-bipartite graph.  If $G$ is weakly $(\ep, 1)$-regular, then $G$ contains a tree $T$ with leaf set $L$ such that 
\begin{enumerate}
\item $|V(T)\cap U|\geq (1-\ep)|U|$, $|V(T)\cap V|\geq (1-\ep)|V|$, and 
\item $|L\cap U|\geq (1-4\ep)|U|$, $|L\cap V|\geq (1-4\ep)|V|$.
\end{enumerate}
\end{lemma}

\begin{proof}
Let $U'$ be a subset of $U$ of size exactly $\floor{3\ep |U|}$ and let $V'$ be a subset of $V$ of size exactly $\floor{3\ep |V|}$.  Let $G'$ be the graph induced by $U', V'$ and suppose first that no component of $G'$ intersects $U'$ in at least $\ep |U|$ vertices or intersects $V'$ in at least $\ep |V|$ vertices.  However, now we may partition $U'$ into sets $U_1$ and $U_2$ and $V'$ into sets $V_1$ and $V_2$ such that $\ep|U|\leq |U_1|\leq 2\ep|U|$, $\ep|V|\leq |V_1|\leq 2\ep |V|$, and there are no edges from $U_1\cup V_1$ to $U_2\cup V_2$; however, this contradicts the fact that $G$ is weakly $(\ep, 1)$-regular.  So suppose that some component $H'$ of $G'$ intersects say $U'$ in at least $\ep |U|$ vertices.  Then since $G$ is weakly-$(\ep, 1)$-regular, we see that all but at most $\ep |V|$ vertices of $V'$ are also in $H'$ and thus there exists a connected subgraph $H'$ of $G$ which intersects each of $U$ and $V$ in at most $3\ep |U|$ and $3\ep |V|$ vertices respectively.  

Finally, using the fact that $G$ is weakly $(\ep, 1)$-regular, we see that all but at most $\ep |U|$ vertices of $U$ and $\ep |V|$ vertices of $V$ have a neighbor in $H'$, and thus $G$ contains a tree in which at least $(1-4\ep)|U|$ vertices of $U$ and at least $(1-4\ep)|V|$ vertices of $V$ are leaves. 

%
\end{proof}

Say that a graph $G$ on $n$ vertices has property $T(r, l, \lambda)$ if every $r$-coloring of $G$ either has a monochromatic tree with at least $(1-\lambda)n$ leaves or $G$ has a $(s,l,n')$-absorbing tree partition for some $2\leq s\leq r$ and $n'\geq (1-\lambda)n$.

\begin{lemma}\label{lem:linsizecommonleaves}
Let $r\geq 1$ and $0<\ep<\frac{1}{10r}$. If $p= \frac{\omega(1)}{n}$, then a.a.s. $\gnp$ has property $T(r,\frac{n}{8er!},6\ep)$.
\end{lemma}

\begin{proof}
Let $G\sim \gnp$ be $r$-colored.  Let $m$ be large enough so that $\frac{\log{m}}{m}<\frac{\ep}{4r}(\frac{1}{er!}-\ep)$ and let $\ep':=\frac{\ep^2}{4r}$.  Apply Lemma \ref{lem:ssrlc} to $G$ with $\ep'$, $m$, and $r$.   Let $\Gamma$ be the $(\ep', p, 1/2r)$-reduced graph obtained.  

By Lemma \ref{lem:nicereduced}, we can pass to a subgraph $\Gamma'\subseteq \Gamma$ with $k':=|\Gamma'|\geq (1-\sqrt{r\ep'})k=(1-\ep/2)k$ and $\delta(\Gamma')\geq (1-2\sqrt{r\ep'})k\ge (1-\ep)k'$. We color the edges of $\Gamma'$ by an arbitrary color $c$ guaranteed by Lemma \ref{lem:weakreginsomec} and recall that this says the $c$-colored edges of $\Gamma'$ represent  $c$-colored weakly-$(\eps, 1)$-regular pairs in $G$.

Now apply Theorem \ref{thm:mindegpartition} with $\ep$ to $\Gamma'$ to get an $(s,l,k')$-absorbing tree partition $\cT$ with $l\geq k'-\frac{2}{\alpha}\log k'\geq  (1-\ep)k'$ (where the second inequality holds by the choice of $m$ and since $k\geq m$) 
if $s=1$ and $l\geq k'/(4er!)$ if $2\leq s\leq r$.  Since the edges of $\Gamma'$ represent weakly-$(\eps, 1)$-regular pairs, we can apply Lemma \ref{lem:leafytree} to each edge of $\cT$
to see that $G$ a.a.s.~has property $T(r,\frac{n}{8er!},6\eps)$; that is, to get the absorbing tree partition $\cT'$ of $G$.  Note that if there are $\tau k'$ leaves in $\cT$,
each of which is a leaf in $s$ different trees, we will get (by Lemma \ref{lem:leafytree}(ii)) at least $\tau k'\cdot (1-4s\ep)\frac{n}{k}\ge (1-5s\ep)\tau n$ leaves in the original graph which are leaves in all $s$ trees. So the total number of leaves in $\cT'$ is at least
\[(1-5s\ep)\tau n \geq 
\begin{cases}
(1-5s\ep)(1-\ep) n\geq (1-6\ep)n &\textrm{ if } s=1 \\
(1-5s\ep)\frac{1}{4er!} n\geq \frac{n}{8er!} &\textrm{ if } 2\leq s\leq r
\end{cases}\]
Also note that since $\cT$ covers the $k'$ vertices of $\Gamma'$, $\cT'$  will cover (by Lemma \ref{lem:leafytree}(i)) at least $k'\cdot(1-\ep)\frac{n}{k} \ge (1-\ep/2)k\cdot (1-\ep)\frac{n}{k}\geq (1-2\ep)n$ vertices of $G$ as desired.

\end{proof}

\section{Lemmas for Random Graphs}\label{ssc:randomlem}
The following lemma follows from standard applications of the Chernoff and union bounds.

\begin{lemma}\label{lem:rgs-common-nbrs}
Let $r\geq 1$.  If $p\ge \bfrac{9r\log n}{n}^{1/r}$, then a.a.s., every set $R$ of $r$ vertices in $G(n,p)$ satisfies $|N^\cap(R)| \ge np^r / 2$. Furthermore, if $p =\omega\of{ \bfrac{\log n}{n}^{1/r} }$, then for any $\eps > 0$, a.a.s.,~every set $R$ of $r$ vertices satisfies
\[(1-\eps)np^r \le |N^\cap(R)| \le (1+\eps)np^r.\]
\end{lemma}

Erd\H{o}s, Palmer, and Robinson \cite{EPR} determined the exact threshold for when the neighborhood of every vertex (of degree at least 2) in $\gnp$ induces a connected subgraph.  We need the following lemma which gives us a bound on the value of $p$ for which the common neighborhood of every set of $r$ vertices in $\gnp$ is non-empty and induces a connected graph.

\begin{lemma}\label{lem:rgs-local-conn}
Let $r\geq 1$.  If $p\ge \bfrac{C\log n}{n}^{1/(r+1)}$ with $C$ sufficiently large, then a.a.s.~in $G\sim G(n,p)$, $G[N^\cap(R)]$ is connected for every set $R$ of $r$ vertices.
\end{lemma}

\begin{proof}
Let $0<\ep\leq 1/2$, let $3(r+1)^2<C'<\frac{C^{1/(r+1)}}{2}$, and suppose $(1-\eps)np^r \le m \le (1+\eps)np^r$. Then
\begin{equation}\label{lowp}
C'\frac{\log m}{m} \leq C'\frac{\log n}{(1-\eps)np^r} \le \frac{C'}{(1-\eps)}\bfrac{\log n}{n}^{1/(r+1)}< \bfrac{C\log n}{n}^{1/(r+1)} \leq p.
\end{equation}

Using \eqref{lowp}, the probability that the subgraph of $G$ induced by a set of size $m$ is disconnected is at most
\begin{align*}\sum_{k=1}^{m/2}\binom{m}{k}(1-p)^{k(m-k)} &\le \sum_{k=1}^{m/2}\exp\of{k\log \of{\frac{me}{k}} - k(m-k)\frac{C'\log m}{m}}\\
&\le \sum_{k=1}^{m/2}\exp\of{k\of{\log m - \frac{C'}{2}\log m}}=o\bfrac{1}{n^{r+1}}.
\end{align*}
There are $O(n^r)$ $r$-sets $R$ for which we must consider $N^\cap(R)$ (which by Lemma \ref{lem:rgs-common-nbrs} satisfy $(1-\eps)np^r \le |N^\cap(R)| \le (1+\eps)np^r$), and thus the expected number of $N(R)$ which induce a disconnected subgraph tends to 0.
\end{proof}

\begin{lemma}\label{lem:gnp-lose-nbrs}
Consider $G\sim \gnp $ with vertex set $V$. Then a.a.s.~for any set $L\subset V$ with $ \frac{80\log n}{p} \le |L| \le n$, all but at most $\frac{9\log n}{p}$ of $v\in V\setminus L$ satisfy  $|N(v, L)| \ge |L|p/2$.
\end{lemma}

\begin{proof}
For a fixed set $L$, we have that $|N(v, L)|\sim Bin(|L|, p)$, so the Chernoff Bound implies that \[\Pr{|N(v, L)| < |L|p/2} \le e^{-\frac{|L|p}{8}}.\]
Call $v$ {\em bad for $L$ } if $|N(v,L)| < |L|p/2$.  
Since $N(v,L)$ and $N(u,L)$ are independent for different vertices $v$ and $u$, we have that the probability that there exists an $L$ with at least $9\log n / p$ many bad vertices is at most
\begin{align*}
\quad\sum_{\ell=80\log n/ p}^{n}\binom{n}{\ell}\binom{n}{9\log n / p}\cdot \exp\of{-\frac{\ell p}{8}\frac{9\log n}{p}}
&\le \sum_{\ell=80\log n/ p}^{n}\exp\of{-\frac{\ell \log n}{8}+\frac{9(\log n)^2}{p}} \\
&\le \exp\of{\log n -\frac{(\log n)^2}{p}}=o(1)
\end{align*}

\end{proof}

The following Lemma will be used to prove Theorem \ref{thm:gnp_threshold_low}.  We prove it here as it may be of independent interest.

\begin{lemma}\label{lem:0statement}~ Let $r\geq 1$.
\begin{enumerate}
\item If $p= \bfrac{r\log n-\omega(1)}{n}^{1/r}$ and $G\sim \gnp$, then a.a.s.~there exists a set $S\subseteq V(G)$ such that $|S|=r$, $S$ is independent, $S$ is not a dominating set, and for all $v\in V(G)$, $\deg(v, S)\leq r-1$.

\item For all $s>r$ with $s\log s=o(\log n)$ and $0 < c < \frac{1}{2\binom sr}$, if $p \le \bfrac{c\log n}{n}^{1/r}$ and $G\sim \gnp$, then a.a.s.~there exists a set $S\subseteq V(G)$ such that $|S|=s$, $S$ is independent, $S$ is not a dominating set, and for all $v\in V(G)$, $\deg(v, S)\leq r-1$.
\end{enumerate}
\end{lemma}

\begin{proof}
We begin with the proof of (ii).  Choose $s$, $c$, and $p$ as in the statement. 
We first show that we can find an independent set $S$ of $s$ vertices such that no $r$ vertices in $S$ have a common neighbor (or equivalently, no $v\in V\setminus S$ has more than $r-1$ neighbors in $S$).  For $S\in \binom{[n]}{s}$ let $X_S$ be the indicator random variable for the event 
\[A_S := \set{\textrm{$S$ is independent and no $r$ vertices in $S$ have a common neighbor}}\]
and let 
\[X = \sum_{S\in\binom{[n]}{s}}X_S.\] 
For a fixed set $S$, let $q$ represent the probability that a vertex $v\not\in S$ has at least $r$ neighbors in $S$. 
Then
\[q := \sum_{k=r}^{s}\binom{s}{k}p^k(1-p)^{s-k}=  \binom{s}{r}p^r\of{ (1-p)^{s-r} + O\of{\sum_{t=1}^{s-r}(sp)^t}} = 
  \binom{s}{r}p^r(1+O(sp)).\]
To see the last two equalities, note that the conditions on $s$ imply that  $sp = o(1)$ and for all $t> 0,$ 
 \[\frac{\binom{s}{r+t}p^{r+t}}{\binom{s}{r}p^r} \le (sp)^{t}.\]

Now we have that
\begin{align*}
\Pr{X_S=1} &= \of{1 - q}^{n-s}(1-p)^{\binom s2}
\end{align*} 
and
\[\E{X} = \binom{n}{s}(1-q)^{n-s}(1-p)^{\binom s2} = \exp\of{s\log \bfrac{ne}{s} - \binom{s}{r}c\log n\opoo + O(1)}\to \infty\]
since $\binom{s}{r}c < 1<s$ and $s\log s = o(\log n)$.

An application of the second moment method will show that such a set exists a.a.s. It suffices to show that $\E{X^2}/\E{X}^2 \le 1+o(1)$ (see e.g.~Corollary 4.3.2 in \cite{AS}). 

\begin{align}
\E{X^2} &= \E{\of{\sum_{S\in\binom{[n]}{s}}X_S}^2}\notag\\
&=\E{X} + \sum_{k=0}^{s-1}\,\sum_{|S_1\cap S_2| = k}\Pr{A_{S_1}\wedge A_{S_2}} \label{eq:sumover}\\
&\le \E{X} +\binom{n}{s}^2(1-q)^{2n - 4s}(1-p)^{2\binom s2}+ \sum_{k=1}^{s-1}O\of{n^{2s-k}}\notag.
\end{align}
The second sum in \eqref{eq:sumover} is over ordered pairs of sets. 
Since $\E{X}^2 = \binom{n}{s}^2(1-q)^{2n-2s}(1-p)^{2\binom s2} = \Omega\of{\bfrac{n}{s}^{2s} e^{-2\binom{s}{r}c\log n} }$, we have \[\E{X^2} / \E{X}^2 \le \frac{1}{\E{X}} + (1-q)^{-2s} + O\of{\frac{s^{2s+1}e^{2\binom{s}{r}c\log n}}{n}} \le  1+o(1)\] since $s\log s = o(\log n)$ and $2\binom{s}{r}c <1$.

Now note that a.a.s.,  $|N^\cup(S)| = O(snp) = o(n)$ and so $S$ is not a dominating set and (ii) is proved. 

The proof of (i) is similar, but we are more careful in the calculation of $\E{X^2}$.  Set $\omega:=\omega(1)$ and suppose $p= \bfrac{r\log n - \omega}{n}^{1/r}$.  We want to prove the existence of a set $S$ of size $r$. 
In this case we simply have $q=p^r$, and thus 
\begin{align*}
\Pr{X_S=1} &= \of{1 - p^r}^{n-r}(1-p)^{\binom r2}
\end{align*} 
and
\begin{align*}
\E{X} &= \binom{n}{r}(1-p^r)^{n-r}(1-p)^{\binom r2}\\ 
&= \exp\of{r\log n - n\cdot \frac{r\log n -\omega}{n} + O(1)} = \exp\of{\omega\opoo}\to \infty.
\end{align*}
For $r$-sets $R_1$, $R_2$ with $|R_1\cap R_2|=k$, we have 
\begin{align*}
\Pr{A_{R_1}\wedge A_{R_2}}&\leq \of{p^k(1-p^{r-k})^2 + (1-p^k)}^{n-2r}
(1-p)^{\binom{r}{2}+k(r-k)+\binom{r-k}{2}}\\
&\leq \of{p^k(1-p^{r-k})^2 + (1-p^k)}^{n-2r}
\end{align*}
Thus
\begin{align*}
\E{X^2} &= \E{X} + \sum_{k=0}^{r-1}\,\sum_{|R_1\cap R_2| = k}\Pr{A_{R_1}\wedge A_{R_2}}
\\
&\le \E{X} + \sum_{k=0}^{r-1}\binom{n}{r}\binom{r}{k}\binom{n-r}{r-k} \of{p^k(1-p^{r-k})^2 + (1-p^k)}^{n-2r}
\\
&\le \E{X} +\binom{n}{r}^2(1-p^r)^{2n - 4r}
+\sum_{k=1}^{r-1}\exp\of{2\omega - k\log n +O(1)}.
\end{align*}
Since $\E{X}^2 =\binom{n}{r}^2(1-p^r)^{2n-2r}(1-p)^{2\binom r2} = \exp\of{2\omega\opoo},$
we again have that \[\E{X^2}/\E{X}^2 \le 1+o(1).\]
\end{proof}

\section{Monochromatic trees in random graphs}
\label{sec:mainthmpf}

\subsection{Upper bounds on the tree cover/partition number}

\begin{theorem}
For all $r\geq 2$, there exists $C\geq r$ such that a.a.s.
\begin{enumerate}
\item  ~if~  $p\ge \bfrac{27\log n}{n}^{1/3}$ ~then~ $\tp_2(G(n,p)) \le 2$, and
\item  ~if~ $p\geq \bfrac{C\log n}{n}^{1/{(r+1)}}$, ~then~ $\tc_r(\gnp )\leq r^2$, and
\item ~if~ $p\ge \bfrac{C\log n}{n}^{1/r}$, then there is a collection of $r$ vertex disjoint monochromatic trees which cover all but at most $9r\log n / p$ vertices.
\end{enumerate}
\end{theorem}

\begin{proof}
Part (ii) follows directly from Proposition \ref{prop:comneigh} and Lemma \ref{lem:rgs-common-nbrs}. 

For part (i), suppose the edges have been colored with colors $1$ and $2$ and consider two vertices $u$ and $v$ with no monochromatic path between them; if there were no such pair, we would have a spanning monochromatic component by the remark of Erd\H{o}s and Rado.  By Lemma \ref{lem:rgs-common-nbrs}, $|N^\cap(\set{u,v})| \ge np^2/2.$ Let $N_{x,c}$ with $x\in \set{u,v}$ and $c\in\set{1,2}$ represent the color $c$ neighbors of vertex $x$ in $N^\cap(\set{u,v})$.  Then we must have that $N_{u,1}\cap N_{v,1} =N_{u,2}\cap N_{v,2} =\emptyset$. Thus $N_{u,1} = N_{v,2}=:A$ and $N_{u,2} = N_{v,1}=:B$. Now by Lemma \ref{lem:rgs-local-conn}, $N^\cap(\set{u,v})$ induces a connected subgraph. If both $A$ and $B$ are non-empty then there is an edge between them, but this edge would give a monochromatic path between $u$ and $v$. Thus wlog, $N^\cap(\set{u,v}) = A$. By Lemma \ref{lem:rgs-common-nbrs}, every vertex in $Z:=V- N^\cap(\set{u,v}) - \set{u,v}$ has at least $np^3/2 \ge 27\log n/2$ many neighbors in $N^\cap(\set{u,v})$. So applying Lemma \ref{lem:Ypartition}, with $k=2$, $Y=N^\cap(\set{u,v})$, and $Z$ as above, we have obtained the desired partition into two monochromatic trees.

In order to prove part (iii),
we will use the method of multiple exposures. Let $p\ge \bfrac{C\log n}{n}^{1/r}$ with $C > 2000e r!r(4r^2)^r$  and let $\wh{p}$ be such that $(1-\wh{p})^{r+1} = (1-p)$. Note that in this case, $p\ge \wh{p} \ge \frac{p}{r+1}.$ 
Then we may view $\gnp$ as $G_0\cup\cdots\cup G_{r}$ where each $G_i \sim G(n,\wh{p})$ for $0\le i\le r$. We will expose the $G_i$ one at a time along with the colors assigned to their edges. Note that if an edge belongs to more than one $G_i$, then when it is revealed a second time, we already know its color. This does not affect our argument.

Let $\alpha = \frac{1}{8er!}$. First we expose $G_0$ and apply Lemma \ref{lem:linsizecommonleaves} with $\epsilon$ as any small constant.  This provides us with an  $(s, \alpha n, n')$-absorbing tree partition $\cT$ and common leaf set $L_0$, such that $1\le s\le r$ and $n'\ge (1-6\ep)n$. If $n'=n$, then we are done. Our goal is now to ``attach" as many of the vertices of $V_0:=V(G)\setminus V(\cT)$ to $L_0$ as possible. Expose $G_1$ (and all the colors assigned to its edges). Apply Lemma \ref{lem:gnp-lose-nbrs} to $G_{1}$ with $L_0$ as $L$. This is possible since $|L_0|\wh{p} \ge \alpha n \wh{p}\ge 80 \log n$. 
Let $V_0' = \set{v\in V_0\,:\, |N(v,L_0)| \ge |L_0|\wh{p}/2}$ and let $V_0'' = V_0\setminus V_0'$. Then by the lemma, $|V_0''| \le 9\log n / p.$
Now if every vertex in $V_0'$ has at least $r\log n$ neighbors in $L_0$ with colors from $[s]$ (note that if $s=r$, then this must be the case), then we may apply Lemma \ref{lem:Ypartition} with $L_0$ as $Y$ and $V_0'$ as $Z$ to get a partition $\set{Y_1, \ldots, Y_s}$ of $L_0$ such that for all $v\in V_0'$ there exists $\ell\in[s]$ such that $N_\ell(v)\cap Y_\ell\neq \emptyset$. By arbitrarily choosing such a $Y_\ell$ for each $v\in V_0'$, we have the desired tree partition of $V\setminus V_0''  $. Otherwise there is a vertex $x_1$ in $V_0'$ satisfying
\[N_j(x_1, L_0)\ge(\frac{|L_0|\wh{p}}{2} - r\log n)/(r-s) \ge \frac{\alpha n \wh{p}}{2r} \]
for some $j\in [r]\setminus [s]$. Without loss of generality, $j=s+1$. Set $L_1:= N_{s+1}(x_1, L_0)$ and $V_1 = V_0'-\set{x_1}$. 
 
Now suppose that for some $1\le i\le r-s$, we have found vertices $\set{x_1,\ldots, x_i} \subset V_0$ and sets $L_{j}$ for $1\le j\le i$ such that $|L_i| \ge \alpha n\bfrac{\wh{p}}{2r}^i$, $L_{i}\subseteq L_{i-1}$ and  $L_{i} \subset N_j(x_j)$ for $1\le j \le i$. Expose $G_{i+1}$ and apply Lemma \ref{lem:gnp-lose-nbrs} to $G_{i+1}$ with $L_i$ as $L$. We may apply the lemma since
\[|L_i|\wh{p} \ge \alpha n \bfrac{\wh{p}}{2r}^{i}\wh{p} \ge \alpha n \bfrac{p}{2r(r+1)}^r\ge  80r\log n.\]  
Let $V_i' = \set{v\in V_i\,:\, |N(v, L_i)| \ge|L_i|\wh{p}/2}$  and let $V_i'' = V_i\setminus V_i'$.  
If every vertex in $V_i'$ has at least $r\log n$ neighbors in $L_i$ with colors from $[s+i]$ (if $s+i = r$, then this is the case by the calculation above), then we may apply Lemma \ref{lem:Ypartition} and we are done as in the base case. 
Otherwise, there is a vertex $x_{i+1} \in V_i'$ satisfying \[N_j(x_{i+1},L_i)\ge (\frac{|L_i|\wh{p}}{2} - r\log n)/(r-(s+i))\ge \alpha n \bfrac{\wh{p}}{2r}^{i+1}\]for some $j\in [r]\setminus[s+i]$. Without loss of generality, $j=s+i+1$. Set $L_{i+1}:= N_{s+i+1}(x_{i+1}, L_i)$ and $V_{i+1} = V_i'-\set{x_{i+1}}$. Thus we will find the desired partition after at most $r-s$ iterations of the above procedure. At each stage we lose at most $9\log n / p$ many vertices and thus we lose at most $9r\log n / p$ in total.
 
\end{proof}

\subsection{Lower bounds on the tree cover number}

\begin{theorem}
For all $r\geq 2$,
\begin{enumerate}
\item  ~if~ $p=\bfrac{r\log n-\omega(1)}{n}^{1/r}$, ~then~ $\tc_r(\gnp )> r$, and
\item ~if~ $p = o\of{ \bfrac{r\log n}{n}^{1/r}}$, ~then~ $\tc_r(\gnp )\to \infty$. 
\end{enumerate}
\end{theorem}

\begin{proof}
(i) We apply Lemma \ref{lem:0statement}~(i) to get an independent set of size $r$ which is not dominating and with no common neighbor.  But then Observation~\ref{obs:need_s_trees} applied with $s=r$ finishes the proof. 

(ii) Choose $s$ as a function of $n$ so that $s\to \infty$, but $s\log s=o(\log n)$.  Similarly to the previous part, the proof follows by applying Lemma~\ref{lem:0statement}~(ii) and Observation~\ref{obs:need_s_trees}.
\end{proof}


\subsection{Large monochromatic components}\label{ssec:lmc}

Finally, we prove that for all $r\geq 2$ and $0<\ep\ll 1/r$, there exists $C>0$ such that if $p\geq \frac{C}{n}$, then a.a.s.~$\tm_r(\gnp )\geq (1-\ep)\frac{n}{r-1}$.

\begin{proof}[Proof of Theorem \ref{thm:randomlargemono}]
Let $r\geq 2$, $0<\ep <\min\{1/2r,1/9\}$, $m\geq 1/\ep^2$, and let $M$ be given by Lemma \ref{lem:ssrlc}.   Choose $C$ sufficiently large for an application of Lemma \ref{lem:edgedist}, let $p\geq \frac{C}{n}$, and let $G_1,\dots, G_r$ be an edge coloring of $G\sim \gnp$.

Apply Lemma \ref{lem:ssrlc} with $\ep^4/r^2$, $m$, and $r$ and then Lemma \ref{lem:nicereduced} to get a ``cleaned-up" reduced graph $\Gamma'$ on $k'\geq (1-\ep^2)k$ vertices with $\delta(\Gamma')\geq (1-2\ep^2)k$.  Apply Theorem \ref{thm:mindeglargemono} to $\Gamma'$ to get a monochromatic tree $\cT$ in $\Gamma'$ with either $|\cT|\geq (1-2\ep^2)k'$ or $\cT$ having at least $(1-4\ep^2)k'/(r-1)$ leaves.  Apply Lemma \ref{lem:leafytree} to each edge of $\cT$ to get the desired tree $T$.  In the second case, note that the  number of leaves of $T$ is at least $$\frac{(1-4\ep^2)k'}{r-1}\cdot(1-4\ep^2)\frac{n}{k}\geq (1-4\ep^2)^2(1-\ep^2)\frac{n}{r-1}\geq (1-\ep)\frac{n}{r-1},$$
where the last inequality holds since $\ep\leq 1/9$.
\end{proof}

\section{Open Problems}\label{sec:conc}

The main open problem is to improve Theorem \ref{thm:gnp_threshold}.(iii) so as to avoid the need for the leftover vertices.  We make the following conjecture and note that for $r\geq 3$ it would be interesting to get any non-trivial bound on $p$ such that $\tp_r(\gnp)\leq r$.

\begin{conjecture}\label{con:cover_sharp} For all $\ep>0$ and $r\geq 1$,
if $p\ge \bfrac{(1+\ep)r\log n}{n}^{1/r}$, then a.a.s.~$\tp_r(\gnp)\leq r$.
\end{conjecture}

Note: While this paper was under review, Kohayakawa, Mota, and Schacht \cite{KMS} proved the $r=2$ case of Conjecture \ref{con:cover_sharp}.

It would be interesting to extend Theorem \ref{thm:2colorcover} to a partition version. 

\begin{conjecture}
For all graphs $G$ on $n$ vertices, if $\delta(G)\geq \frac{2n-5}{3}$, then $\tp_2(G)\leq 2$. 
\end{conjecture}

For the following conjecture, Theorem \ref{thm:2colorcover} provides the $r=2$ case, and for the $r=1$ case, note that if $\delta(G)\geq \frac{n-1}{2}$, then $G$ is connected, i.e.~$\tp_1(G)=\tc_1(G)=1$.  Furthermore, Example \ref{obs:tcrlb} shows that this is best possible if true.

\begin{conjecture}\label{prob:cov}
For all $r\geq 1$, if $G$ is a graph on $n$ vertices with $\delta(G)\geq \frac{r(n-r-1)+1}{r+1}$, then $\tc_r(G)\leq r$.
\end{conjecture}

Regarding the distinct colors variant of these problems, we make the following conjecture which is true for $r=1$ as above, and for $r=2$ by Letzter's result \cite{L2}.  Furthermore, Example \ref{ex:tcrG} shows that this is best possible if true.  

\begin{conjecture}\label{mindegtrees}
Let $r\geq 1$.  If $\delta(G)\geq (1-\frac{1}{2^r})n$, then $G$ has property $\cT\cC_r$ ($\cT\cP_r$).
\end{conjecture}

Finally, in Theorem \ref{thm:mindegpartition} we prove that if an $r$-colored graph $G$ has sufficiently large minimum degree, then $G$ can be partitioned into $r$ monochromatic trees, each of which implicitly has many leaves.  What about partitioning into trees with few leaves?

\begin{problem}
For all $r\geq 2$, sufficiently small $\ep>0$, and sufficiently large $n_0$, if $G$ is a graph on $n\geq n_0$ vertices with $\delta(G)\geq (1-\ep)n$, then in every $r$-coloring of $G$ there exists a partition of $G$ into $O(r)$ monochromatic trees so that each tree has $O(1)$ leaves.
\end{problem}

\section{Acknowledgements}

We thank Rajko Nenadov, Frank Mousset, Nemanja \v{S}kori\'c and independently Hi\d{\^{e}}p H\'an for drawing our attention to an error in an earlier version of this paper related to Theorem \ref{thm:gnp_threshold}.  We also thank the referees for making many useful comments which helped us improve the organization of the paper.

\end{document}